\documentclass[letterpaper, 10 pt, conference,twocolumn]{ieeeconf}
\usepackage[utf8]{inputenc}

\usepackage{amsmath}
\usepackage{amssymb}
\usepackage{graphicx}
\usepackage{tikz}
\usepackage{comment}
\usepackage{pgfplots}
\pgfplotsset{compat=1.9}
\usepackage{amsthm}
\usepackage[sorting=none]{biblatex}
\usepackage{xcolor}
\newcommand{\E}[2]{\ensuremath{\mathbb{E}_{#2}[#1]}}
\newcommand{\Cov}[1]{\ensuremath{\text{Cov}[#1]}}
\newcommand{\R}[1]{\ensuremath{\mathbb{R}^{#1}}}
\newcommand{\tr}[1]{\ensuremath{\text{tr}\Big(#1 \Big)}}

\newcommand{\bs}[1]{\boldsymbol{#1}}

\newtheorem{proposition}{Proposition}

\newtheorem{problem}{Problem}
\newtheorem{remark}{Remark}

\DeclareMathOperator*{\Minimize}{Minimize}

\DeclareMathOperator{\st}{subject\:to}

\usepackage{blindtext}
\IEEEoverridecommandlockouts 

\title{Covariance Steering of Discrete-Time Stochastic Linear Systems Based on Distribution Distance Terminal Costs}

\author{Isin M. Balci and Efstathios Bakolas\thanks{I. M. Balci (graduate student) and E. Bakolas (Associate Professor) are with the Department of Aerospace Engineering
and Engineering Mechanics, The University of Texas at Austin,
Austin, Texas 78712-1221, USA, Emails: isinmertbalci@utexas.edu; bakolas@austin.utexas.edu} \thanks{This research  has been supported  in part by NSF  awards ECCS-1924790 and CMMI-1937957.}}

\begin{document}

\maketitle

\begin{abstract}
    We consider a class of stochastic optimal control problems for discrete-time stochastic linear systems which seek for control policies that will steer the probability distribution of the terminal state of the system close to a desired Gaussian distribution. In our problem formulation, the closeness between the terminal state distribution and the desired (goal) distribution is measured in terms of the squared Wasserstein distance which is associated with a corresponding terminal cost term. We recast the stochastic optimal control problem as a finite-dimensional nonlinear program and we show that its performance index can be expressed as the difference of two convex functions. This representation of the performance index allows us to find local minimizers of the original nonlinear program via the so-called convex-concave procedure~\cite{ccp}. Subsequently, we consider a similar problem but this time we use a terminal cost that corresponds to the KL divergence. Finally, we present non-trivial numerical simulations to demonstrate the proposed techniques and compare them in terms of computation time. 
\end{abstract}
\section{Introduction}

We consider covariance steering problems for discrete-time stochastic linear systems in which, however, the constraints on the terminal state covariance are enforced indirectly by means of appropriate terminal costs. Specifically we consider the problem of steering the state of a stochastic system, which is originally drawn from a given Gaussian distribution, to a terminal state whose distribution is ``close'' to a desired (prescribed) Gaussian distribution, where the closeness between the two distributions is measured in terms of the squared Wasserstein distance or the Kullback-Leibler divergence. We show that the resulting problems can be reduced to tractable optimization problems which can be solved efficiently if one exploits their structure.

\textit{Literature Review:} The main focus of the first attempts to study covariance steering problems \cite{min-var-control,all-cov-controllers, lyapunov-covariance-controllers} was on finding stabilizing controllers that drive the state covariance to a desired positive definite matrix asymptotically (infinite-horizon case). Finite-horizon covariance control problems for continuous-time linear systems were recently studied in \cite{chen-optimal-1, chen-optimal-2, chen-optimal-3,bakolas-integral-quadratic}. 
Covariance steering problems for discrete-time systems are also receiving significant attention at present. In \cite{bakolas-auto2018}, the constrained covariance steering problem is recast as a finite dimensional convex optimization problem based on a semidefinite relaxation of the constraint on the terminal state covariance. Covariance steering problems with convex chance constraints are studied in \cite{okamoto2018chance}. 

In the previously discussed references, the specifications on the terminal state covariance correspond to hard constraints which often lead to difficult problems (for instance, the analytic solution to the covariance steering problem presented in \cite{chen-optimal-1} is only valid for the special case in which the input and noise channels coincide). An alternative problem formulation, which has inspired this paper, is presented in \cite{halder2016finite} in which a terminal cost is used as a ``soft'' constraint on the terminal state covariance. The latter cost corresponds to the squared Wasserstein distance between a desired state distribution and the ``actual'' terminal state distribution. The latter formulation leads to a standard two-point boundary value problem which can be solved by means of indirect shooting methods. It is well known that the success of such methods relies on knowledge of good initial guesses and thus, in general, a systematic process for the computation of the solution to the class of covariance steering problems proposed in \cite{halder2016finite} with soft terminal constraints is still missing.

\textit{Main Contribution:} We first formulate the covariance steering problem as a stochastic optimal control problem in which the requirement on the terminal state covariance is encoded in a terminal cost term (``soft constraint''). Similarly with \cite{halder2016finite}, we consider the case in which the terminal cost corresponds to the squared Wasserstein distance between the actual terminal state distribution and the desired Gaussian distribution but in contrast with the latter reference, we consider the discrete-time case. First, we recast this stochastic optimal control problem as a (deterministic) nonlinear program by utilizing an affine state feedback control policy parametrization (the control input at each stage is an affine function of the history of visited states). Then, we show that the performance index of the nonlinear program can be expressed as the difference of two convex functions by using a suitable bilinear transformation of the decision variables. To the best of our knowledge, this is the first paper that shows that covariance steering problems can be formulated as a difference of convex functions program (DCP). By leveraging this fact, one can find local minimizers of the nonlinear program via efficient techniques such as the so-called convex-concave procedure~ (CCP)~\cite{ccp, dccp-boyd}. The CCP is an iterative procedure  which can compute local minimizers of non-convex optimization problems which correspond to DCP based on successive convexifications. Exploiting this extra structure of the problem reduces its complexity and allows us to use convex optimization solvers which in turn leads to improved scalability and numerical efficiency.

Finally, we consider the same class of problems when the terminal cost corresponds to the Kullback-Leibler divergence, which is used as a measure of the closeness between the terminal state distribution and the goal distribution (one can also consider different generalized distance metrics between the two distributions; a review of distance metrics on probability distributions can be found in \cite{prob-metrics-review}). Even though the resulting nonlinear program does not corresponds to a DCP, we show empirically, that one can compute its local minimizers by using interior-point methods for nonlinear programs.

 \textit{Outline:} The rest of the paper is organized as follows. Section II presents the problem formulation. In Section III, we show that when the terminal cost is the squared Wasserstein distance, the covariance steering problem can be associated with a difference of convex functions program. In Section IV, we provide an alternative problem formulation in which the terminal cost corresponds to the KL divergence. In Section V, we present numerical simulations. Finally, Section VI concludes the paper with a summary of remarks and future research directions.

\section{Problem Formulation}
\subsection{Notation}
We denote by $\mathbb{R}^n$ the set of $n$-dimensional real vectors and by $\mathbb{R}$ and $\mathbb{R}^{+}$ (resp., $\mathbb{R}^{++}$) the set of real numbers and non-negative (resp., strictly positive) real numbers, respectively. The sets of non-negative and strictly positive integers are denoted by $\mathbb{Z}^{+}$ and $\mathbb{Z}^{++}$, respectively. We denote by $\mathbb{E}[\cdot]$
the expectation operator. Given a random vector $x$, we denote its mean vector and covariance matrix by $\mathbb{E}[x]$ and $\operatorname{Cov}[x]$, respectively. The space of $n \times n$ symmetric matrices is denoted by $\mathbb{S}_n$ and the cone of positive semi-definite (definite) symmetric matrices by $\mathbb{S}_n^+$ ($\mathbb{S}_n^{++}$). The trace of a square matrix is denoted as $\operatorname{tr(\cdot)}$. The transpose of a matrix $A \in \mathbb{R}^{n \times m}$ is denoted by $A^{\mathrm{T}}$ and its nuclear norm by $\| A \|_{*}$ where $\| A \|_{*}:=\operatorname{tr}((A^{\mathrm{T}} A)^{1/2})$. The block diagonal matrix formed by $n$ matrices $A_1, \dots, A_n$ is denoted by $\operatorname{blkdiag}(A_1, \dots, A_n)$. The zero matrix is denoted as $\boldsymbol{0}$ whereas the identity matrix as $\boldsymbol{I}$. We write $ x \sim \mathcal{N}(\mu, S)$ to denote that $x$ is a Gaussian random vector with mean $\mu \in \mathbb{R}^n$ and covariance $S \in \mathbb{S}_n^{++}$. 

\subsection{Distance Between Probability Distributions}
In this paper, we will formulate stochastic optimal control problems with terminal costs that measure the closeness between the final state distribution and a desired probability distribution. In particular, we will consider two different distribution (generalized) distance functions, namely, the Wasserstein distance and the KL divergence.

\subsubsection{Wasserstein Distance Between Two Distributions}
The Wasserstein distance between two probability measures is a valid distance metric (in the strict mathematical sense) because it satisfies all of the properties of a metric. Given two random vectors $x_1$, $x_2$ over $\R{n}$ with probability density functions $\rho_1 , \rho_2$, their squared Wasserstein distance is defined as follows:
\begin{equation}
    W^2(\rho_1, \rho_2) := \inf_{\rho \in \mathcal{P}(\rho_1, \rho_2)} \E{\|x_1 - x_2 \|_2^2}{y},
\end{equation}
where $y := [x_1,~ x_2]^{\mathrm{T}}$ and has a probability density function (pdf) $\rho : \R{2n} \rightarrow \R{+}$. Furthermore, $\mathcal{P}(\rho_1, \rho_2)$ denotes the set of all probability distributions over $\mathbb{R}^{2n}$ with finite second moments and marginals $\rho_1$ and $\rho_2$ on $x_1$ and $x_2$, respectively. 

If $x_i \sim \mathcal{N}(\mu_i, S_i)$ for $i={1,2}$ where $\mu_i\in\mathbb{R}^n$ and $S_i \in \mathbb{S}^{++}_n$, then the squared Wasserstein distance is given by~\cite{wasserstein-closed-form}
\begin{align}\label{wass-cost}
    W^2(\rho_1, \rho_2) & = \|\mu_1 - \mu_2 \|_2^2 \nonumber \\
    &~~~+ \tr{S_1 + S_2 - 2 (S_2^{1/2} S_1 S_2^{1/2})^{1/2}}.
\end{align}

\subsubsection{Kullback-Leibler Divergence}
The KL divergence is not a metric in the strict mathematical sense (it does not enjoy the symmetry property) but it is often used to compute the ``distance'' between two distributions because of its ease of computation. In particular, given two probability distributions with density functions $\rho_1(x)$ and $\rho_2(x)$, their KL divergence is defined as:
\begin{equation}
\mathrm{KL}(\rho_1 || \rho_2) = \int \rho_1(x) \log\bigg(\frac{\rho_1(x)}{\rho_2(x)}\bigg) \mathrm{d}x.
\end{equation}
where $\rho_2(x) > 0$ over the domain of integration. 

When $\rho_1$ and $\rho_2$ correspond to the densities of two Gaussian distributions $\mathcal{N}(\mu_1,S_1)$ and $\mathcal{N}(\mu_2,S_2)$ where $\mu_i \in \mathbb{R}^{n}$ and $S_i \in \mathbb{S}^{++}_{n}$ is given by 
\begin{align}\label{KL-div}
  & \mathrm{KL}(\rho_1 || \rho_2) =
    (1/2) \big[\tr{S_2^{-1} S_1} + (\mu_2 - \mu_1)^{\mathrm{T}} S_2^{-1} (\mu_2 - \mu_1)
    \nonumber \\
  &  ~~~~~~~\qquad \qquad \qquad  - n + \log\big(\det (S_2)/\det   (S_1)\big)\big].
\end{align}



\subsection{Problem Statement}
We consider an uncertain system whose dynamics is described by the following discrete-time stochastic linear state space model:
\begin{equation}\label{linear-system}
    x_{k+1} = A_k x_k + B_k u_k + G_k w_k,~~~~\forall k \in \mathbb{Z}^{+},
\end{equation}
where $\{x_k\}_{k\in \mathbb{Z}^{+}}$ is the state (random) process over $\mathbb{R}^{n_x}$, $\{u_k \}_{k\in \mathbb{Z}^{+}}$ is the input process over $\mathbb{R}^{n_u}$ and $\{w_k\}_{k\in \mathbb{Z}^{+} }$ is the noise (random) process over $\mathbb{R}^{n_w}$. In particular, $\{w_k\}_{k\in \mathbb{Z}^{+} }$ corresponds to a white Gaussian noise process with $\mathbb{E}[w_k]=0$ and $\mathbb{E}[w_k w_m^{\top}] = \delta (k,m) \bs{I}$, where $\delta(k,m) = 1$ when $k=m$ and $\delta(k,m)=0$, otherwise. We also assume that the initial state $x_0 \sim \mathcal{N}(\mu_0, S_0)$ and that $x_0$ and $\{w_k\}$ are mutually independent, which implies that $\mathbb{E}[x_0 w_k^{\top}]=\mathbf{0}$ for all $k \in \mathbb{Z}^{+}$.

Our objective is to drive the uncertain state of the system \eqref{linear-system} from its given initial distribution to a terminal distribution which is close to a desired terminal Gaussian probability distribution $\mathcal{N}(\mu_d, S_d)$, where $\mu_d \in \mathbb{R}^n$ and $S_d\in\mathbb{S}^{++}_n$ are given, at a given finite time while minimizing a relevant performance index. Next, we provide the precise formulation of our problem.

\begin{problem}\label{prob-def-1}
Let $\mu_0, \mu_f \in \R{n_x}$, $S_{0}, S_f \in \mathbb{S}_{n_x}^{++}$, $\lambda > 0$ and $ N \in \mathbb{Z}^{++}$ be given. In addition, let $\Pi$ denote the set of all admissible control policies $\pi :=\{m_0(\cdot), \dots, m_{N-1}(\cdot) \}$ for system \eqref{linear-system}, with $u_k = m_k(X^k)$ where $X^k$ denotes the (finite) sequence of states visited up to stage $t=k$, that is, $X^k := \{x_0, x_1, \dots x_k \}$, and $m_k(X^k)$ are measurable functions of the elements of $X^k$, for $k=0, \dots, N-1$. Then, find a control policy $\pi^* \in \Pi$ 
that solves the following stochastic optimal control problem:
\begin{subequations}\label{opt1}
\begin{align}
&\Minimize_{\pi \in \Pi}          &\qquad& \mathbb{E}\Bigg[{\sum_{k=0}^{N-1} u^{\mathrm{T}}_k u_k}\Bigg] + \lambda \varphi (\rho_N, \rho_d)\\
& \st &      & x_{k+1} = A_k x_k + B_k u_k + G_k w_k \\
&                  &      & x_0 \sim \mathcal{N}(\mu_0, S_0) 
\end{align}
\end{subequations}
where $\rho_d$ is the pdf of the Gaussian probability distribution $\mathcal{N}(\mu_f, S_f)$ (desired state distribution), $\rho_N$ is the pdf of the terminal state $x(N)$, and $\varphi(\rho_N, \rho_d)$ denotes the (generalized) distance between the probability distributions of the desired state and the actual terminal state of the system. In particular, $\varphi(\rho_N, \rho_d)= W^2(\rho_N, \rho_d)$ or $\varphi(\rho_N, \rho_d)= \mathrm{KL}(\rho_N || \rho_d)$.
\end{problem}

In order to associate Problem \ref{prob-def-1} with a tractable, finite-dimensional optimization problem, we only consider admissible control policies that correspond to sequences of control laws $m_k(\cdot)$ which are affine functions of the state history:
\begin{align}\label{controller}
    m_k(X^k) = \sum_{i=0}^{k} K(k, i) \big(x_i-\bar{x}_i\big) + u_{\mathrm{ff}}(k),
\end{align}
where $\bar{x}_i=\E{x_i}{}$. Next, we show the main steps for recasting the Problem~1, whose decision variable corresponds to the control policy $\pi$, as an optimization problem whose decision variables are the controller parameters $u_{\mathrm{ff}}(k) \in \R{n_u}$ and $K(k, j) \in \R{n_u \times n_x}$, $\forall k \geq j \in  \{0, \dots , N-1 \}$.


\section{Covariance Steering Based on a Wasserstein Distance Terminal Cost}
In this section, we will show that Problem~1 when $\varphi(\rho_N, \rho_d)= W^2(\rho_N, \rho_d)$ can be associated with a difference of convex function program (DCP), that is, a nonlinear program whose performance index is equal to the difference of two convex functions. This will allow us to efficiently compute local minimizers of Problem~1 by means of heuristic and easily implementable algorithms, such as the convex-concave procedure \cite{ccp}. It is worth mentioning that the set of objective functions which can be expressed as the difference of convex functions is dense in the set of continuous functions; moreover, every twice differentiable function can be represented as the difference of convex functions~\cite{dc-prog-overview}. However, there is no systematic process that is guaranteed to find such a representation for a given function of interest except for special classes of functions. 

Next, we recast Problem~\ref{prob-def-1} as a finite-dimensional optimization problem. To this aim, we express the state $x_k$ in terms of a finite-dimensional decision variable. In particular, by propagating forward in time the state of the discrete-time stochastic system~\eqref{linear-system} and using the control policy parametrization given in \eqref{controller}, we can express $x_k$ as a function of $x_0$, $\{ u_i \}_{i=0}^{k-1}$ and $\{ w_i \}_{i=0}^{k-1}$ as follows:
\begin{equation}\label{xkfunc}
    x_k = \Phi(k, 0) x_0 + \sum_{i=0}^{k-1} \Phi(k,i) B_{i} u_{i} + \sum_{i=0}^{k-1} \Phi(k,i) G_i w_i,
\end{equation}
where $\Phi(k, n) \triangleq A_{k-1} \dots A_{n}$, $\Phi(n,n) = \bs{I}$ with $ k \geq n$ for $k,n \in \mathbb{Z}^{+}$. Now, let us define the following quantities:
\begin{subequations}\label{vars}
\begin{align}
    \boldsymbol{x} & := [x(0)^{\mathrm{T}}, x(1)^{\mathrm{T}}, \dots, x(N)^{\mathrm{T}}]^{\mathrm{T}} \in \R{n_x  (N+1)},\\
     \boldsymbol{u} & := [u(0)^{\mathrm{T}}, u(1)^{\mathrm{T}}, \dots, u(N-1)^{\mathrm{T}}]^{\mathrm{T}} \in \R{n_u N}, \\
     \boldsymbol{w} & := [w(0)^{\mathrm{T}}, w(1)^{\mathrm{T}}, \dots, w(N-1)^{\mathrm{T}}]^{\mathrm{T}} \in \R{n_w N}.
 \end{align}
\end{subequations}
By using equations~\eqref{xkfunc}-\eqref{vars}, it follows that
\begin{equation}\label{opt-form}
    \boldsymbol{x} = \boldsymbol{\Gamma} x_0 + \boldsymbol{H_u} \boldsymbol{u} + \boldsymbol{H_w} \boldsymbol{w},
\end{equation}
where 
\begin{equation}
    \bs{\Gamma}  := [\bs{I}~\Phi(1,0)~\Phi(2,0)~\dots~ \Phi(N,0)],
\end{equation}
\begin{align}
    \bs{H_u} & := \begin{bmatrix}
    \bs{0} & \bs{0} & \dots & \bs{0} \\
    B_0 & \bs{0} & \dots & \bs{0} \\
    \Phi(2,1)B_0 & B_1 & \dots & \bs{0} \\
    \vdots & \vdots & \vdots & \vdots \\
    \Phi(N, 1)B_0 & \Phi(N, 2)B_1 & \dots & B_{N-1}  
    \end{bmatrix}, \label{Hudef}
\end{align}
and $\bs{H_w}$ is defined similarly, after replacing the matrices $B_i$ in \eqref{Hudef} with the matrices $G_i$. One can refer to~\cite{bakolas-auto2018} for the details on the derivation of \eqref{opt-form}-\eqref{Hudef}.

Because the performance index of Problem \ref{prob-def-1} consists of a terminal cost term, we will use the following equation:
\begin{equation}
    x(N) = \boldsymbol{F} \boldsymbol{x},~~~\boldsymbol{F} := [\boldsymbol{0} \cdots ~\boldsymbol{0} ~ \boldsymbol{I}],
\end{equation}
%
to recover $x(N)$ from $\boldsymbol{x}$.

Given the particular affine parametrization of the control policy as in \eqref{controller} and the fact that the initial state is assumed to be a Gaussian (random) vector, it follows that the states of the system in the subsequent stages will also be Gaussian (random) vectors. In addition, we obtain
\begin{align}
    \boldsymbol{u} = \boldsymbol{K} (\boldsymbol{x} - \boldsymbol{\Bar{x}}) + \boldsymbol{u}_{\mathrm{ff}}, 
\end{align}
where $\boldsymbol{\Bar{x}} := \E{\boldsymbol{x}}{}$, $\bs{u}_{\mathrm{ff}} := [u_{\mathrm{ff}}^{\mathrm{T}}(0), \dots, u_{\mathrm{ff}}^{\mathrm{T}}(N-1)]^{\mathrm{T}}$ and
\begin{align}\label{Kdef}
    \bs{K} &:= \begin{bmatrix}
    K(0,0) & \bs{0} & \dots &  \bs{0} \\
    K(1,0) & K(1,1) & \dots &  \bs{0} \\
    K(2,0) & K(2,1) & \dots &  \bs{0} \\
    \vdots & \vdots & \vdots &  \vdots \\
    K(N-1,0) & K(N-1,1) &  \dots & \bs{0}
    \end{bmatrix}.
\end{align}
We proceed with the derivation of the expression of the performance index of Problem~1 in terms of the new decision variables. To this aim, we write $\sum_{k=0}^{N-1} u_k^\mathrm{T} u_k = \boldsymbol{u}^T \boldsymbol{u}$, which in view of basic properties of trace operator and \eqref{controller} gives
\begin{align}\label{raw-Euud}
    \E{\boldsymbol{u}^\mathrm{T} \boldsymbol{u}}{} & = \E{\operatorname{tr}( \boldsymbol{u}\boldsymbol{u}^\mathrm{T})}{} \nonumber \\
    & = \E{\operatorname{tr}(     (\boldsymbol{K} (\boldsymbol{x} - \boldsymbol{\Bar{x}}) + \boldsymbol{u}_{\mathrm{ff}})  
    (\boldsymbol{K} (\boldsymbol{x} - \boldsymbol{\Bar{x}}) + \boldsymbol{u}_{\mathrm{ff}})^\mathrm{T}    )}{} \nonumber \\
   & = 
    \operatorname{tr}(\boldsymbol{K} \E{\Tilde{\boldsymbol{x}} \Tilde{\boldsymbol{x}}^\mathrm{T}}{} \boldsymbol{K}^\mathrm{T}) + \| \boldsymbol{u}_{\mathrm{ff}}\|_2^2,
\end{align}
where $\Tilde{\boldsymbol{x}} := \boldsymbol{x} - \boldsymbol{\Bar{x}}$ and in the derivation of the last equality, we have used the fact that $\boldsymbol{u}_{\mathrm{ff}}$ is a deterministic quantity.

For the computation of $\operatorname{Cov}[\bs{x}] = \E{\Tilde{\boldsymbol{x}} \Tilde{\boldsymbol{x}}^\mathrm{T}}{}$, we first have to compute $\Bar{\boldsymbol{x}} = \E{\bs{x}}{}$. By taking expectation of both sides of \eqref{opt-form}, we obtain:
\begin{align}\label{Ex}
    \E{\boldsymbol{x}}{} & = \E{\boldsymbol{\Gamma} x_0 + \boldsymbol{H_u} (\bs{K} (\bs{x} - \bs{\Bar{x}}) + \bs{u}_{\mathrm{ff}}) + \bs{H_w w}}{} \nonumber \\
&= \bs{\Gamma} \mu_0 + \bs{H_u u}_{\mathrm{ff}}.
\end{align}
After some simple algebraic manipulations, we get:
\begin{equation}
    \Tilde{\bs{x}} = (\bs{I} - \bs{H_u K})^{-1} ( \bs{\Gamma} (x_0 -\mu_0) + \bs{H_w w}).
\end{equation}
Let $\Bar{\bs{K}} := (\bs{I} - \bs{H_u K})^{-1}$ and $\Bar{x}_0 := x_0 - \mu_0$. We obtain:
\begin{equation}\label{Covx}
    \E{\Tilde{\bs{x}} \Tilde{\bs{x}}^\mathrm{T}}{} = 
    \Bar{\bs{K}} (\Gamma S_0 \Gamma^\mathrm{T} + \bs{H_w}S_{\bs{w}} \bs{H_w}^\mathrm{T}) \Bar{\bs{K}}^\mathrm{T}.
\end{equation}
From \eqref{Ex} and \eqref{Covx},
we can obtain the following expressions for $\mu_N :=\E{x(N)}{}$ and $S_N :=\Cov{x(N)}$: 
\begin{subequations}\label{mufandSf}
\begin{align}
    \mu_N & = \boldsymbol{F} (\bs{\Gamma} \mu_0 + \bs{H_u u}_{\mathrm{ff}}), \label{ExN}\\
     S_N & = \boldsymbol{F} (\bs{I} - \bs{H_u K})^{-1} \Tilde{S} (\bs{I} - \bs{H_u K})^\mathrm{-T} \boldsymbol{F}^\mathrm{T}, \label{CovxN}
\end{align}
\end{subequations}
where $\Tilde{S} = (\Gamma S_0 \Gamma^\mathrm{T} + \bs{H_w}S_{\bs{w}} \bs{H_w}^\mathrm{T})$ and $S_{\bs{w}} = \E{\bs{w} \bs{w}^\mathrm{T}}{}$. 
By plugging \eqref{Covx} into \eqref{raw-Euud}, we have: 
\begin{align}\label{Euu}
    \E{\bs{u}^\mathrm{T} \bs{u}}{} & = \operatorname{tr} (\bs{K} (\bs{I} - \bs{H_u K})^{-1} \Tilde{S} (\bs{I} - \bs{H_u K})^\mathrm{-T} \bs{K}^\mathrm{T}) \nonumber \\
    &~~~~ +  \| \bs{u}_{\mathrm{ff}} \|^2.
\end{align}

After plugging the expressions of $\mu_N$ and $S_N$ in \eqref{ExN} and \eqref{CovxN} into the expression of $ W^2(\rho_N, \rho_d)$ in the case of Gaussian distributions, which is given in \eqref{wass-cost}, we get:
\begin{align}\label{wass-K}
    & W^2(\rho_N, \rho_d) =~ \lVert \bs{F} (\bs{\Gamma} \mu_0 + \bs{H_u} \bs{u}_{\mathrm{ff}}) - \mu_d \rVert^2_2 \nonumber \\
    & + \operatorname{tr} (\bs{F} (\bs{I} - \bs{H_u K})^{-1} \Tilde{S} 
    \bs{F} (\bs{I} - \bs{H_u K})^\mathrm{-T} \bs{F}^\mathrm{T} +  S_d) \nonumber \\
    & -2 \operatorname{tr}(\sqrt{S_d} ~ \times \nonumber \\
    & \quad (\bs{F} (\bs{I} - \bs{H_u K})^{-1} \Tilde{S} 
    \bs{F} (\bs{I} - \bs{H_u K})^\mathrm{-T} \bs{F}^\mathrm{T})^{1/2} \nonumber \\
    &\quad \times \sqrt{S_d}).
\end{align}

At this point, we propose to apply a variable transformation, which was first proposed in \cite{boyd-cvx-control} and later used for covariance steering problems in \cite{bakolas-auto2018}, to convexify the optimization problem. In particular, we a new  transformed variable, $\bs{\Theta}$, which is defined as follows:
\begin{subequations}\label{transformations}
\begin{gather}
    \bs{\Theta} := \bs{K}(\bs{I} - \bs{H_u K})^{-1} =: \varphi (\bs{K}) \label{eqx}\\
    \bs{K} := (\bs{I} + \bs{H_u \Theta})^{-1} \bs{\Theta} =: \phi(\bs{\Theta}). \label{eqinvtrans}
\end{gather}
Furthermore, by using the identity $(\bs{I} + \bs{P})^{-1} = \bs{I} - \bs{P} (\bs{I} + \bs{P})^{-1}$, we obtain:
\begin{align}
    (\bs{I} - \bs{H_u K})^{-1} & = \bs{I} + \bs{H_u K} (\bs{I} - \bs{H_u K})^{-1} \nonumber \\ 
     & = (\bs{I} + \bs{H_u \Theta}).
\end{align}
\end{subequations}
As is shown in \cite{boyd-cvx-control}, the functions $\phi (\cdot)$ and $\varphi (\cdot)$ determine a bijective transformation, that is,  $\phi(\cdot) = \varphi^{-1}(\cdot)$ and vice versa. Therefore, the right hand sides of equations (\ref{Euu}) and (\ref{wass-K}) can be expressed equivalently in terms of transformed variables (\ref{transformations}) as follows:
\begin{align}
    & \E{\bs{u}^\mathrm{T} \bs{u}}{}  = \operatorname{tr} (\bs{\Theta} \Bar{S} \bs{\Theta}^\mathrm{T}) + \bs{u}_{\mathrm{ff}}^\mathrm{T} \bs{u}_{\mathrm{ff}} \\
   & W^2  = \lVert \bs{F}(\bs{\Gamma}\mu_0 + \bs{H_u} \bs{u}_{\mathrm{ff}}) -\mu_d \rVert_2^2 \nonumber \\
    & + \operatorname{tr}(\bs{F} (\bs{I} + \bs{H_u \Theta}) \Tilde{S} (\bs{I} + \bs{H_u \Theta})^\mathrm{T} \bs{F}^\mathrm{T}) \nonumber \\
    & - 2 \operatorname{tr} ((\sqrt{S_d} \bs{F} (\bs{I} + \bs{H_u \Theta}) \Tilde{S} (\bs{I} + \bs{H_u \Theta})^\mathrm{T} \bs{F}^\mathrm{T} \sqrt{S_d})^{1/2}) \nonumber \\
    & + \operatorname{tr}(S_d).
\end{align}

\begin{remark}
It should be noted that $\bs{K}$ is a block lower triangular matrix whose last $n_x$ columns are equal to $0$. If we examine equation \eqref{eqinvtrans}, we observe that $(\bs{I} - \bs{H_u K})^{-1}$ is block lower triangular since $\bs{H_u}$ is also block lower triangular, which implies that $(\bs{I} - \bs{H_u K})^{-1}$ is well defined. Finally, left multiplication of $(\bs{I} - \bs{H_u K})^{-1}$ with $\bs{K}$ gives $\bs{\Theta}$, which is also a block lower triangular matrix with the same dimension as $\bs{K}$. The reader can refer \cite{bakolas-auto2018, boyd-cvx-control} for more details. An important observation is that the new decision variable $\bs{\Theta}$ should have the same structure as $\bs{K}$ for the control policy to maintain causality. 
\end{remark}

Finally, the performance index of Problem \ref{prob-def-1} can be expressed in terms of the decision variables $\bs{u}_{\mathrm{ff}}$ and $\bs{\Theta}$. Let us denote this function as $J(\bs{u}_{\mathrm{ff}}, \bs{\Theta})$, where 
\begin{equation}\label{objective-problem2}
    J(\bs{u}_{\mathrm{ff}}, \bs{\Theta}) = J_1(\bs{u}_{\mathrm{ff}}) +J_2(\bs{\Theta}) + J_3(\bs{\Theta}) - J_4(\bs{\Theta}),
\end{equation}
with 
\begin{subequations}\label{Jfunc}
\begin{align}
J_1(\bs{u}_{\mathrm{ff}}) & :=  \lVert\bs{u}_{\mathrm{ff}} \rVert_2^2  +\lambda  \lVert \bs{F}(\bs{\Gamma}\mu_0  + \bs{H_u} \bs{u}_{\mathrm{ff}}) -\mu_d \rVert_2^2, \\
J_2(\bs{\Theta}) &:= \operatorname{tr} (\bs{\Theta} \Bar{S} \bs{\Theta}^\mathrm{T}), \\
J_3(\bs{\Theta}) & := \lambda \operatorname{tr}(\bs{F} (\bs{I} + \bs{H_u \Theta}) \Tilde{S} (\bs{I} + \bs{H_u \Theta})^\mathrm{T} \bs{F}^\mathrm{T}) 
     \nonumber \\ & ~~~~+ \operatorname{tr}(S_d),\\
J_4(\bs{\Theta}) & := 2 \lambda \operatorname{tr}((\sqrt{S_d} \bs{F} (\bs{I} + \bs{H_u \Theta}) \nonumber \\
& ~~~\qquad~~~ \times \Tilde{S} (\bs{I} + \bs{H_u \Theta})^\mathrm{T} \bs{F}^\mathrm{T} \sqrt{S_d})^{1/2}).
\end{align}
\end{subequations}
Thus, Problem~\ref{prob-def-1} can be reduced to the following optimization problem:
\begin{problem}\label{problem2}
Let $\mu_0, \mu_d \in \mathbb{R}^{n_x}$, $S_0, S_d \in \mathbb{S}_{n_x}^{++}$, $N \in \mathbb{Z}^{++}$ and $\{ A_k, B_k, G_k \}_{k=0}^{N}$, where $A_k \in \mathbb{R}^{n_x \times n_x}$, $B_k \in \mathbb{R}^{n_x \times n_u}$ and $G_k \in \mathbb{R}^{n_x \times n_w}$, be given. Find a pair $(\bs{u}_\mathrm{ff}^{\star}, \bs{\Theta}^{\star})$, where $\bs{\Theta}^{\star}$ is a block lower triangular matrix in $\mathbb{R}^{n_uN \times n_x(N+1)}$ and $\bs{u}_{\mathrm{ff}}^{\star} \in \mathbb{R}^{n_uN}$, that minimizes the objective function $J(\bs{u}_\mathrm{ff}, \bs{\Theta})$, which is defined in \eqref{objective-problem2}-\eqref{Jfunc}.
\end{problem}

%

\begin{proposition}\label{diff-convex-prop}
Let $\lambda \in \R{++}$ be given. Then. the functions $J_1$, $J_2$, $J_3$ and $J_4$, which are defined in \eqref{Jfunc}, are convex and thus Problem~\ref{problem2} corresponds to a difference of convex functions program (DCP).
\end{proposition}
\begin{proof}
The proof of convexity of the functions $J_1(\cdot)$, $J_2(\cdot)$ and $J_3(\cdot)$ can be found in \cite{bakolas-auto2018}. For the convexity of $J_4(\cdot)$, we need to define the functions $g(\bs{\Theta}) := (\sqrt{S_d} \bs{F} (\bs{I} + \bs{H_u \Theta}) \Tilde{V} \Tilde{D}^{1/2})^T$, where $\Tilde{V}^\mathrm{T} \Tilde{D} \Tilde{V}$ is the eigenvalue decomposition of $\Tilde{S}$, and $f(\bs{\mathcal{A}}) := \text{tr}((\bs{\mathcal{A}}^\mathrm{T} \bs{\mathcal{A}})^{1/2}) = \| \bs{\mathcal{A}} \|_{*}$. Clearly, $g(\cdot)$ is an affine function. In addition, $f(\cdot)$ corresponds to the nuclear norm, which is a valid matrix norm \cite{horn2012matrix} and thus, $f(\cdot)$ is a convex function. Finally, $J_4(\bs{\Theta})$ is convex as the composition of the convex function $f(\cdot)$ with the affine function $g(\cdot)$.
\end{proof}
\begin{remark}
Proposition~\ref{diff-convex-prop} implies that Problem \ref{prob-def-1} can be reduced to a DCP, whose (local) minimizers can be found by means of the so-called convex-concave procedure \cite{ccp, dccp-boyd} which is known to be efficient and robust in practice.
\end{remark}

\section{NLP formulation for KL Divergence Terminal Cost}
If we consider Problem~1 when the terminal cost $\varphi(\rho_N, \rho_d)= \mathrm{KL}(\rho_N || \rho_d)$, then we will arrive at a nonlinear program (NLP) similar to Problem~\ref{problem2}. However, using the variable transformations given in \eqref{transformations} will not yield a DCP as in the case with $\varphi(\rho_N, \rho_d)= W^2(\rho_N, \rho_d)$. Thus, using state history feedback will not necessarily help us associate the covariance steering problem (Problem \ref{prob-def-1}) to a tractable optimization problem. We will instead consider a \textit{memoryless} state feedback (affine) controller in the form: 
\begin{equation}\label{eq:newcontroller}
    u_k = K(k) \big(x_k - \Bar{x}_k\big) + u_{\mathrm{ff}}(k),
\end{equation}
where $\Bar{x}_k = \E{x_k}{}$. Independent of the choice of the controller form, we can express the running cost term of the performance index of Problem~\ref{prob-def-1} as in \eqref{Euu} and also obtain expressions for the mean and variance of the final state $x(N)$ as in \eqref{ExN} and \eqref{CovxN}. The only difference will be in the matrix $\bs{K}$ which is now defined as $\bs{K} := [\operatorname{blkdiag}(K(0), K(1), \dots, K(N-1)) , \bs{0}]$, which is significantly more sparse than the previous case. Because the closed-loop system is linear (given the structure of the controller given in \eqref{eq:newcontroller}), the final state $x(N)$ will be a Gaussian random variable $x(N) \sim \mathcal{N}(\mu_N, S_N)$ and thus we can recover an NLP using the expression for the KL divergence given in \eqref{KL-div}. Since the KL divergence is not symmetric, the final objective function will depend on the order of $\rho_N$ and $\rho_d$. We take $\varphi(\rho_N, \rho_d)= \mathrm{KL}(\rho_N || \rho_d)$. Thus, the objective function can be expressed as follows by plugging \eqref{Euu} and \eqref{mufandSf} into \eqref{KL-div}:
\begin{align}\label{KL-div-obj-func}
   & J(\bs{K}, \bs{u}_{\mathrm{ff}})  =  \| \bs{u}_{\mathrm{ff}}\|^2 \nonumber \\ 
    &+ \operatorname{tr} (\bs{K} (\bs{I} - \bs{H_u K})^{-1} \Tilde{S} (\bs{I} - \bs{H_u K})^\mathrm{-T} \bs{K}^\mathrm{T}) \nonumber \\
    &+  (\lambda/2)\Big[ \operatorname{tr} \big(S_d^{-1} \boldsymbol{F} (\bs{I} - \bs{H_u K})^{-1} \Tilde{S} (\bs{I} - \bs{H_u K})^\mathrm{-T} \boldsymbol{F}^\mathrm{T}  \big) \nonumber  \\
    &+ (\mu_d - \boldsymbol{F} (\bs{\Gamma} \mu_0 + \bs{H_u} \bs{u}_{\mathrm{ff}}))^\mathrm{T} S_d^{-1}
    (\mu_d - \boldsymbol{F} (\bs{\Gamma} \mu_0 + \bs{H}_u \bs{u}_{\mathrm{ff}})) \nonumber \\
    &- n_x + \log(\det S_d) \nonumber \\
    &- \log\big(\det (\boldsymbol{F} (\bs{I} - \bs{H_u K})^{-1} \Tilde{S} (\bs{I} - \bs{H_u K})^\mathrm{-T} \boldsymbol{F}^\mathrm{T})\big)
    \Big].
\end{align}

In this case, Problem~1 reduces to the following optimization problem:
\begin{problem}\label{problem3}
Let $\mu_0, \mu_d \in \mathbb{R}^{n_x}$, $S_0, S_d \in \mathbb{S}_{n_x}^{++}$, $N \in \mathbb{Z}^{++}$ and $\{ A_k, B_k, G_k \}_{k=0}^{N}$, where $A_k \in \mathbb{R}^{n_x \times n_x}$, $B_k \in \mathbb{R}^{n_x \times n_u}$ and $G_k \in \mathbb{R}^{n_x \times n_w}$, be given. Find a pair $(\bs{K}^{\star}, \bs{u}_{\mathrm{ff}}^{\star})$, where $\bs{K}^{\star} := [\operatorname{blkdiag}(K^{\star}(0), K^{\star}(1), \dots, K^{\star}(N-1)) , \bs{0}]$ where $K^{\star}(i)\in \mathbb{R}^{n_u \times n_x}$, for $i\in \{0,\dots, N-1\}$, and $\bs{u}_{\mathrm{ff}}^{\star} \in \mathbb{R}^{n_uN}$, that minimizes the objective function $J(\bs{K}, \bs{u}_{\mathrm{ff}})$ defined in \eqref{KL-div-obj-func}.
\end{problem}

Because the objective function given in \eqref{KL-div-obj-func} is not convex in $(\bs{K}, \bs{u}_{\mathrm{ff}})$, Problem \ref{problem3} corresponds to a non-convex NLP, in general. In addition, Problem \ref{problem3} does not correspond to a DCP, but local minimizers of this problem can still be computed by using nonlinear interior point methods and solvers such as IPOPT\cite{ipopt} and the scipy optimization package~\cite{virtanen2020scipy}, which are readily available.


\section{Numerical Experiments}
In this section, we present numerical experiments where we used the convex-concave procedure (CCP) with MOSEK \cite{mosek2010mosek} to solve Problem \ref{problem2} and CVXPY \cite{diamond2016cvxpy} for modeling of convexified subproblems. To solve Problem \ref{problem3} which is a nonlinear program, we used the scipy optimization \cite{virtanen2020scipy} implementation of the L-BFGS-B algorithm. We consider the linear state space model \eqref{linear-system} with $A_k = \big[ \begin{smallmatrix} 1 & \Delta t \\ 0 & 1\end{smallmatrix} \big]$, $B_k = [ 0~\Delta t]^{\mathrm{T}}$, $G_k = \bs{I}$, $w_k \sim \mathcal{N}(\mathcal{\bs{0}, \gamma \bs{I}})$, $\forall k \in \mathbb{Z}^{+}$. We also took $x_0 \sim \mathcal{N}(\mu_0, S_0)$, $\mu_0 = [0, 1]^{T}$, $S_0 = 10  \bs{I}$, $\mu_d = [10, 12]^\mathrm{T}$, $S_d = \bs{I}$, $\Delta t = 1$. In addition, $N \in \{10,20,30,40,50\}$ and $\gamma \in \{ 1, 0.5 \}$ are chosen for different experiments to compare computation time. 

Figure \ref{confidence-evolution-compare} illustrates the evolution of the state distribution of the system. We use $\lambda = 10.0$ for the Wasserstein distance case and $\lambda=70.0$ for the KL divergence case for scaling purposes. The noise intensity parameter $\gamma = 1$ and the problem horizon $N=20$ in both experiments. The final state covariance matrices are $\big[ \begin{smallmatrix} 2.81 & 0.19 \\ 0.19 & 1.98  \end{smallmatrix} \big]$ for the Wasserstein distance and $\big[\begin{smallmatrix} 3.65 & 0.06 \\ 0.06 & 2.21 \end{smallmatrix}\big]$ for the KL divergence. Since both problems are non-convex, the obtained solutions are expected to depend on the initial guess. However, repeating the experiments with different initial guesses did not change the final cost and the covariance matrices significantly even though the control policy parameters did change. 
\begin{figure}[ht]
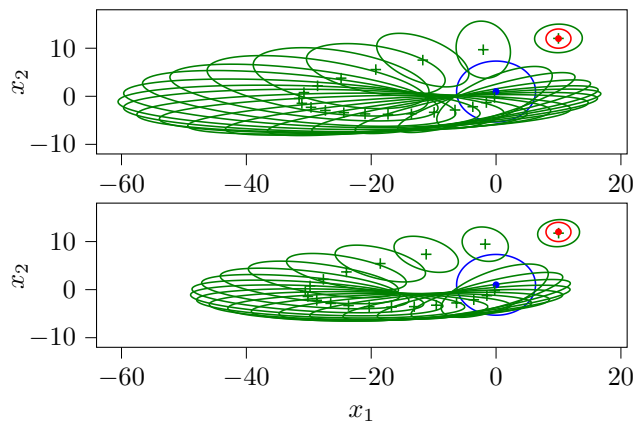

    \begin{minipage}{1.0\linewidth}
        \centering
        \input{plots/evolution-KL}
    \end{minipage}
    \begin{minipage}{1.0\linewidth}
        \centering
        \input{plots/evolution-wasserstein}
    \end{minipage}
    \caption{Evolution of 2-$\sigma$ confidence ellipses (in green) for the controlled system based on KL divergence (top) and Wasserstein distance (bottom) terminal cost functions. The blue ellipses correspond to the 2-$\sigma$ confidence ellipses of the initial state, whereas the red ellipses to the desired distribution.}
    \label{confidence-evolution-compare}
\end{figure}

In Figure \ref{default-N40}, sample paths of the controlled system are shown for $N=40$. We observe that the optimal control policy allows the spread of trajectories (uncertainty) to ``grow" in the beginning and tries to reduce it down towards the end of the time horizon. This result is expected given that the state covariance is not penalized in the running cost term of the performance index in Problem~1 whereas the uncertainty in the control input is penalized by the term $\operatorname{tr}(\bs{\Theta} \Bar{S} \bs{\Theta})$.

\begin{figure}[ht]
    \centering
    \includegraphics[width=\linewidth]{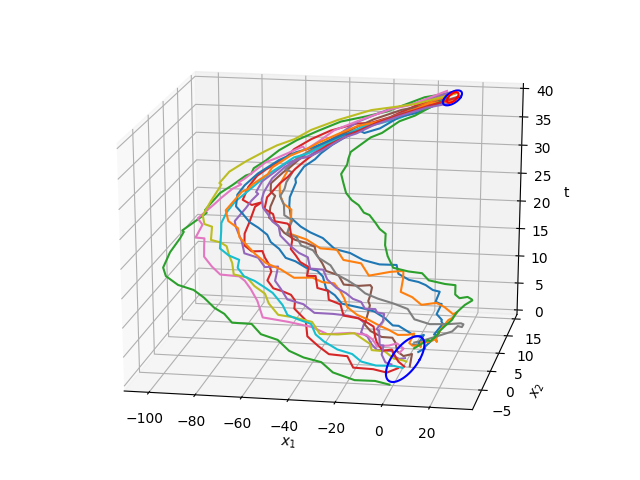}
    \caption{15 sample paths of controlled system with Wasserstein distance terminal cost where blue ellipses are 2-$\sigma$ confidence regions for initial state and final state and red ellipse is the 2-$\sigma$ confidence region of desired distribution. ($\gamma=1$, $\lambda=10.0$, $N=40$).}
    \label{default-N40}
\end{figure}

In Table \ref{computation-table}, we compare the computation time of the NLP solver \cite{virtanen2020scipy} and our CCP based approach for different problem instances with different values for the noise intensity parameter $\gamma$ and the problem horizon $N$. In our simulations, we used the termination condition $(f_k - f_{k-1})/f_k \leq \epsilon$ where $f_k$ is the value of objective function at the $k$th iteration and $\epsilon$ is the convergence tolerance which was taken to be $10^{-5}$. We observe that our approach reduces the computation time significantly in all cases. 

\begin{table}[ht]
    \centering
    \caption{Computation time (in seconds) for different problem instances for the Wasserstein Distance Terminal Cost}
    \label{computation-table}
    \begin{tabular}{c|c|c|c|c|c}
          $\gamma = 1$&  N=10 & N=20 & N=30 & N=40 & N=50\\
          \hline
         NLP & 7.88 & 44.30 & 120.93 & 348.39 & 643.65 \\
         CCP & 0.93 & 7.65 & 12.81 & 32.85 & 68.72 \\
         \hline
         \hline
         $\gamma = .5$&  N=10 & N=20 & N=30 & N=40 & N=50\\
         \hline
         NLP & 18.01 & 28.57 & 209.93 & 510.14 & 907.40 \\
         CCP & 2.89 & 17.11 & 53.68 & 156.29 & 314.52
    \end{tabular}
\end{table}

\section{Conclusion}
We have addressed the covariance steering problem with soft terminal constraints based on two different problem formulations in which the terminal cost is associated with either the squared Wasserstein distance or the KL divergence between the terminal state distribution and a desired distribution. We have shown that in the case with the squared Wasserstein distance terminal cost, the proposed covariance steering problem reduces to a DCP which can be solved efficiently by the so-called convex-concave procedure along with convex optimization solvers. Our numerical experiments have shown that our approach reduces significantly the computation time compared to off-the-shelf solvers. In our future work, we plan to extend our approach to covariance steering problems for nonlinear stochastic systems.


\printbibliography[title=references]
\end{document}